 \documentclass[draft]{article}

\usepackage{amsmath,amsfonts,amsthm,amssymb,amscd,cancel,color}
\usepackage{enumitem}
\usepackage{verbatim}
\usepackage[dvipdfmx]{graphicx}
\usepackage{ulem,color}

\theoremstyle{plain}
\newtheorem{theorem}{Theorem}[section]

\newtheorem{proposition}[theorem]{Proposition}

\theoremstyle{definition}

\theoremstyle{remark}
\newtheorem{remark}[theorem]{Remark}

\newcommand{\R}{{\mathbb R}}

\newcommand{\Z}{{\mathbb Z}}

\def\im{{\rm i}}

\newcommand{\C}{\mathbb{C}}

\def\({\left(}
\def\){\right)}
\def\<{\left\langle}
\def\>{\right\rangle}

\numberwithin{equation}{section}

\begin{document}

\title{A remark on uniform Strichartz estimate for quantum walks on 1D lattice}

\author{Takumi Aso, Masaya Maeda}
\maketitle

\begin{abstract}
In this short note, we study quantum walks (QWs) on one dimensional lattice $\delta \mathbb{Z}$.
Following Hong-Yang \cite{HY19DCDS}, we prove Strichartz estimates for QWs independent of the lattice width $\delta$.
\end{abstract}

\section{Introduction}
Discrete time quantum walks, or quantum walks (QWs) in short, are space-time discrete unitary dynamical systems first appeared as the space-time discretization of $1+1$ dimensional Dirac equation related to Feynman path integral \cite{FH10Book}.
From the end of 20th century, QWs started to attract interest of researchers in many fields of mathematics and physics, see \cite{Konno08LNM,Portugal13Book} and reference therein.
Also, for interesting discussion about the generalization of Quantum walks, see \cite{SS1802.01837, Sako1902.02479}.

The purpose of this note is to combine the the study of QWs as dispersive equation \cite{MSSSSdis,MSSSS18DCDS,MSSSS19JPC} and the study of the continuous limit of QWs, see \cite{MS20RMP} and reference therein.
In particular, we study the Strichartz estimate of QWs in the continuous limit setting and obtain ``uniform" Strichartz estimates, where ``uniform" means that the inequality is independent of lattice size $\delta>0$.

To state our result precisely, we prepare several notations.
First, Pauli matrices are given by
\begin{align}\label{pauli}
\sigma_0:=\begin{pmatrix}1&0\\0&1\end{pmatrix},\  \sigma_1:=\begin{pmatrix}0 &1\\1&0\end{pmatrix},\  \sigma_2:=\begin{pmatrix}0&-\im \\ \im &0\end{pmatrix},\  \sigma_3:=\begin{pmatrix}1 &0\\0&-1\end{pmatrix}.
\end{align} 
The ``shift operator" $\mathcal{S}_\delta$ is given by
\begin{align*}
\(\mathcal{S}_\delta u\)(x):=\begin{pmatrix}
u_1(x-\delta)\\ u_2(x+\delta)
\end{pmatrix}, \ \text{for}\ u=\begin{pmatrix}
u_1\\ u_2
\end{pmatrix},
\end{align*}
and ``coin operator" $\mathcal{C}_\delta $ is given by
\begin{align*}
\(\mathcal{C}_\delta u\)(x) :=
\begin{pmatrix}
\cos (\delta m ) & -\im \sin(\delta  m )\\ -\im \sin(\delta  m ) & \cos (\delta m ) 
\end{pmatrix} 
u(x)  = e^{-\im \delta  m  \sigma_1}u(x).
\end{align*}
where $ m \in \R$ is a constant and $u={}^t\!(u_1\ u_2):\delta \Z\to \C^2$.

Given the shift and coin operators, we define 
\begin{align}\label{def:qwg}
\mathcal{U}_\delta&:=\mathcal{S}_\delta\mathcal{C}_\delta,\quad
U_\delta(t)u:=\mathcal{U}_\delta^{t/\delta}u,\ t\in \delta\Z.
\end{align}

\begin{remark}
	We can consider more general families of coin operators.
	However, for the simplicity of exposition, we chose to consider only the above family of coin operators.
\end{remark}

\begin{remark}
	The definition of $U_\delta(t)$, in particular the reason we take $t\in \delta\Z$ and $U_\delta(t):=\mathcal{U}_\delta^{t/\delta}$ instead of taking $t\in \Z$ and $U_\delta(t):=\mathcal{U}_\delta^{t}$, 
	is inspired by the continuous limit.
	Indeed, formally $\mathcal{S}_\delta \sim 1-\delta \sigma_3 \partial_x$ and $C_\delta \sim 1 -\im \delta  m  \sigma_1$, we see $\mathcal{U}_\delta\sim 1-\delta(\sigma_3 \partial_x + \im  m  \sigma_1)$.
	Thus, we have
	\begin{align*}
	\im \partial_t u(t)\sim \im \delta^{-1}(u(t+\delta)-u(t))=\im\delta^{-1}(\mathcal{U}_\delta u(t)-u(t))\sim -\im \sigma_3 \partial_x u(t)+  m  \sigma_1 u(t),
	\end{align*}
	which is a $1+1$ dimensional (free) Dirac equation up to $O(\delta)$ error.
	
\end{remark}

We further define several notations to state our result.

\begin{itemize}
	\item $\<a\>:=(1+|a|^2)^{1/2}$.
	\item For $u={}^t\!(u_1\ u_2)\in \C^2$, we set $\|u\|_{\C^2}:=\(|u_1|^2+|u_2|^2\)^{1/2}$.
	\item For $p,q\in [1,\infty]$, we set
	\begin{align*}
	l^p_\delta&:=\left\{u:\delta\Z\to \C^2\ |\ \|u\|_{l^p_\delta}:=\(\delta \sum_{x\in \delta \Z} \|u(x)\|_{\C^2}^p\)^{1/p}<\infty\right\},\\
	l^p_\delta l^q_\delta&:=\left\{f:\delta \Z\times \delta\Z \to \C^2\ |\ \|f\|_{l^p_\delta l^q_\delta}:=\(\sum_{t\in \delta \Z} \delta  \(\sum_{x\in \delta \Z}\delta \|f(t,x)\|_{\C^2}^q\)^{p/q}\)^{1/p} \right\},
	\end{align*}
	with the standard modification for the cases $p=\infty$ or $q=\infty$.
	\item We define the Fourier transform $\mathcal{F}_\delta $ and the inverse Fourier transform by
	\begin{align*}
	\mathcal{F}_\delta u = \frac{\delta}{\sqrt{2\pi}}\sum_{x\in \delta \Z}e^{-\im x \xi}u(x),\quad
	\mathcal{F}_\delta^{-1} v   = \frac{1}{\sqrt{2\pi}}\int_{-\pi/\delta}^{\pi/\delta}e^{\im x \xi}v(\xi)\,d\xi.
	\end{align*}
	\item
	Given $p:[-\pi/\delta,\pi/\delta] \to \R$, we define $p(\mathcal{D}_\delta)$ by
	$$
	\(p(\mathcal{D}_\delta)u\)(x):=\mathcal{F}_\delta^{-1} \(p(\xi) \(\mathcal{F}_\delta u\)(\xi)\)(x).
	$$
	\item We say $(p,q)\in [2,\infty]\times [2,\infty]$ is an admissible pair if $3p^{-1}+q^{-1}=2^{-1}$.
	We say $(p,q)\in [2,\infty]\times [2,\infty]$ is a continuous admissible pair if $2p^{-1}+q^{-1}=2^{-1}$.

	\item For $p\in [1,\infty]$, we denote the $p'$ will mean the H\"older conjugate of $p$, i.e.\ $\frac{1}{p}+\frac{1}{p'}=1$.
	\item We write $a\lesssim b$ by meaning $a\leq C b$ where $C>0$ is a constant independent of quantities we are concerning. 
	In particular, in this note, the implicit constant never depends on $\delta$.
	If $a\lesssim b$ and $b\lesssim a$ then we write $a\sim b$.
	
\end{itemize}

We are now in the position to state our main result.

\begin{theorem}\label{thm:main}
	Let $\delta \in (0,1]$ and $u: \delta \Z \to \C^2$, $f:\delta\Z\times \delta \Z \to \C^2$.
	Let $(p,q)$ and $(\tilde{p},\tilde{q})$ be admissible pairs.
	Then, we have
	\begin{align}
	\| U_\delta(t)u\|_{l^p_\delta l^q_\delta} &\lesssim \| |\mathcal{D}_\delta|^{1/p}\<\mathcal{D}_\delta\>^{3/p}u\|_{l^2_\delta},\label{eq:homest}\\
	\| \sum_{s\in [0,t]\cap \delta\Z}U_\delta(t-s) f \|_{l^p_\delta l^q_\delta} & \lesssim \| |\mathcal{D}_\delta|^{1/p+1/\tilde{p}}\<\mathcal{D}_\delta\>^{3/p+3/\tilde{p}}f\|_{l^{\tilde{p}'}_\delta l^{\tilde{q}'}_\delta} \label{eq:inhomest}.
	\end{align}
\end{theorem}

We compare our result with the known results by Hong-Yang \cite{HY19DCDS} who proved similar result for discrete Schr\"odinger equations and discrete Klein-Gordon equations.
Here, we briefly recall the results for discrete Schr\"odinger equations.
First, discrete Schr\"odinger equations on $1$D lattice $\delta\Z$ is given by
\begin{align*}
\im \partial_t u(x) = -\Delta_\delta u(x):=\delta^{-2}(2u(x)-u(x-\delta)-u(x+\delta)),\ u:\R\times \delta\Z\to \C.
\end{align*}
For fixed $\delta$, the Strichartz estimate
\begin{align*}
\|e^{\im t \Delta_\delta} u\|_{L^p l^q_\delta}\leq C_\delta \|u_0\|_{l^2_\delta},\ (p,q)\text{ is admissible},
\end{align*}
was proved by Stefanov-Kevrekidis \cite{SK05N} where the  constant $C_\delta=C\delta^{-1/p}$ blows up as $\delta\to 0$.
In general, one expect that solutions of discretized equation converge to the solutions of the original equation.
This is true for Schr\"odinger equation for fixed time interval.
However, as we saw above, the Strichartz estimate for discrete Schr\"odinger equation, which control the global behavior of solutions, do not converge to the Strichartz estimate of the continous Schr\"odinger equation.
This is because of the lattice resonance, which occurs because the dispersive relation for a equation of lattice always have degeneracy.
Indeed, since the dispersive relation for an equation on $\delta \Z$ is a function $p_\delta(\xi)$ on a torus, if $p_\delta$ is smooth, then there exists a point such that $p_\delta''(\xi)=0$.
Also, we can observe the difference between discrete and continous Strichartz estimate from the difference of admissible pair.

Hong-Yang observed that the difference of the scale can be absorbed by fractal difference operator $|\mathcal{D}_\delta|^{1/p}$ and obtained the result:
\begin{align*}
\|e^{\im t \Delta_\delta} u\|_{L^p l^q_\delta}\lesssim \||\mathcal{D}_\delta|^{1/p}u_0\|_{l^2_\delta},\ (p,q)\text{ is admissible}.
\end{align*}
This estimate is compatible with the estimate for the continuous Schr\"odinger equations.
Hong-Yang also prove similar estimates for discrete Klein-Gordon equations.

We now discuss QWs.
In \cite{MSSSS18DCDS}, the Strichartz estimate for QWs was given by
\begin{align}\label{eq:Stz1}
\|U_\delta(t)u\|_{l^p_\delta l^q_\delta} \lesssim C_\delta \|u\|_{l^2_\delta},\quad \ (p,q)\text{ is admissible},
\end{align}
where $C_\delta\to \infty$ as $\delta\to 0$.
Since QWs are space-time discrete Dirac equations, we compare \eqref{eq:Stz1} with the the Strichartz estimates of the spacetime continuous Dirac equations:
\begin{align*}
\|e^{-\im t H}u\|_{L^pL^q}\lesssim \| \<\partial_x \>^{3/p}u\|_{L^2},\quad (p,q)\text{ is continous admissible},,
\end{align*}
where $H=-\im\sigma_3\partial_x +  m  \sigma_1$ (for the proof see \cite{NS11Book} and reference therein).
To make $(p,q)$ to be admissible, using Sobolev inequality we have
\begin{align*}
\|e^{\im t H}u\|_{L^pL^q}\lesssim \| |\partial_x|^{1/p}\<\partial_x \>^{3/p}u\|_{L^2},\quad \ (p,q)\text{ is admissible}.
\end{align*}
We now see that estimate \eqref{eq:homest} is compatible with the above estimate of Dirac equation.
For the inhomogeneous estimate \eqref{eq:inhomest} there is also the similar correspondence.

In the next section, we prove Theorem \ref{thm:main} following Hong-Yang \cite{HY19DCDS}.
There are two difference between the discrete Schr\"odinger equation treated in \cite{HY19DCDS} and QWs.
The first is that discrete Schr\"odinger/Klein-Gordon equations are discretized only in space, while QWs are discretized in spacetime.
The second difference is that QWs have more complicated dispersive relation than discrete Schr\"odinger/Klein-Gordon equations.
For the first problem we follow \cite{MSSSS18DCDS}.
For the 2nd problem, we carefully estimate the dispersive relation for low and high frequency region separately to get the optimal result.

\section{Proof of Theorem \ref{thm:main}}

Let $\phi \in C_0^\infty$ satisfy $\chi_{[-1,1]}(x)\leq \phi(x)\leq \chi_{[-2,2]}(x)$ for all $x\in\R$, where $\chi_A$ is the characteristic function of $A\subset \R$.
For $\lambda>0$, we set $\psi_{\delta,\lambda}\in C^\infty([-\pi/\delta,\pi/\delta])$ by $\psi_{\delta,\lambda}(\xi):=\psi(\xi/\lambda)$ where $\psi(x):=\phi(x)-\phi(2x)$.
We note that since $\mathrm{supp}\psi_{\delta,\lambda} \subset \([-2\lambda,-\lambda/2]\cup [\lambda/2,2\lambda]\)\cap [-\pi/\delta,\pi/\delta]$, we have $\psi_{\delta,\lambda}= 0$ a.e.\ if $\lambda \geq 2\pi/\delta$ and for $\lambda < 2\pi/\delta$, we have
\begin{align*}
\xi \in \mathrm{supp}\psi_{\delta,\lambda}\ \Rightarrow \ |\xi| \sim \lambda. 
\end{align*}
Using this $\psi_{\delta,\lambda}$, we define the Littlewood-Paley projection operators by
\begin{align}\label{eq:def:LW}
P_\lambda:=P_{\delta,\lambda}:=\psi_{\delta,\lambda}(\mathcal{D}_\delta).
\end{align}
We further set $\tilde \psi \in C_0^\infty$ to satisfy $0\not\in \mathrm{supp}\tilde{\psi}$ and $\tilde \psi(\xi)=1$ for $\xi \in \mathrm{supp}\psi$ and set $\tilde \psi_{\delta,\lambda}$ and $\tilde P_N$ as $\tilde{\psi}_{\delta,\lambda}(\xi)=\psi(\xi/\lambda)$ and \eqref{eq:def:LW}.
In particular, we have $P_\lambda = P_\lambda \tilde{P}_\lambda$.

The main ingredient of the proof of Theorem \ref{thm:main} is the following proposition.

\begin{proposition}\label{prop:main}
	Let $\delta\in (0,1]$ and $\lambda\in (0,2\pi/\delta )$.
	Then, we have
	\begin{align}\label{eq:mainest}
	\|U_\delta(t)P_\lambda u\|_{l^\infty_\delta} \lesssim \lambda^{1/3} \<\lambda\>t^{-1/3}\|u\|_{l^1_\delta}.
	\end{align}
\end{proposition}

\begin{proof}
As in \cite{MSSSS18DCDS}, by Fourier transformation, we have
\begin{align}
U_\delta(t) P_\lambda u_0(x)=\sum_{\mathfrak{s}\in \{\pm\}}\(I_{\delta,\lambda,\mathfrak{s}}*u\)(x),\label{eq:Uconv}
\end{align}
where $I*u(x):=h\sum_{x\in \delta\Z} I(x-y)u(y)$ and
\begin{align}
I_{\delta,\lambda ,\mathfrak{s}}&:=\frac{1}{2\pi}\int_{-\pi/\delta}^{\pi/\delta}e^{\im\(\mathfrak{s}   p_\delta(\xi) t + x\xi\)}Q_{\delta,\mathfrak{s}}(\xi)\psi_{\delta,\lambda }(\xi)\,d\xi,\label{eq:I}\\
p_\delta(\xi)&:=\delta^{-1}\mathrm{arccos}\(\cos(\delta  m )\cos(\delta \xi)\),
\end{align}
where $Q_{\delta,\pm}$ are $2\times 2$ matrices, which can be computed explicitly.

\begin{remark}
	Notice that 
	$
	p_\delta(\xi)=\sqrt{ m ^2+\xi^2}+O(\delta)
	$
	for fixed $\xi$, which tells us that QWs have the dispersion relation similar to Dirac equations (and Klein-Gordon equations) in the continuous limit.
\end{remark}

For $\mathfrak{s}=\pm$, $i,j=1,2$, since $Q_{\delta,\pm}$ depends on $\xi$ only through $\delta\xi$, 
by $\|f(\delta \xi)\|_{L^\infty([-\pi/\delta,\pi/\delta])}=\|f\|_{L^\infty([-\pi,\pi])}$ and $\| \(f(\delta \xi)\)'\|_{L^2([-\pi/\delta,\pi/\delta])}=\|\delta f'(\delta \xi)\|_{L^1([-\pi/\delta,\pi/\delta])}=\|f'\|_{L^1([-\pi,\pi])} $, we see that 
\begin{align}\label{eq:estQ}
\| Q_{\delta,\mathfrak{s},i,j}\|_{L^\infty[-\pi/\delta,\pi/\delta]}+\| Q_{\delta,\mathfrak{s},i,j}'\|_{L^1[-\pi/\delta,\pi/\delta]}\lesssim 1,
\end{align}
where $Q_{\delta,\mathfrak{s},i,j}$ is the $i,j$ component of $Q_{\delta,\mathfrak{s}}$.
Since we are intending to use van der Corput Lemma, we record the derivatives of $p_\delta$:
\begin{align}
p_\delta'(\xi)&=\frac{\cos(\delta  m )\sin(\delta \xi)}{(1-\cos^2(\delta  m )\cos^2(\delta \xi))^{1/2}}, \label{eq:p'}\\
p_\delta''(\xi)&=\frac{\delta\cos(\delta  m )\sin^2(\delta  m )\cos(\delta \xi)}{(1-\cos^2(\delta  m )\cos^2(\delta \xi))^{3/2}}, \label{eq:p''}\\
p_\delta'''(\xi)&=\frac{\delta^2\cos(\delta  m )\sin^2(\delta  m )(1+2\cos^2(\delta  m )\cos^2(\delta \xi))\sin(\delta \xi)}{(1-\cos^2(\delta  m )\cos^2(\delta \xi))^{5/2}}, \label{eq:p'''}.
\end{align}


To prove \eqref{eq:mainest}, we consider two cases $\lambda \leq 1/2\delta$ and $1/2\delta<\lambda$.
First, if $\lambda \leq 1/2\delta$, then by the elementary inequality $\sqrt{1-a^2}\leq \cos a$ for $|a|\leq 1$, we have
\begin{align*}
1-\cos^2(\delta  m )\cos^2(\delta \xi) \leq  \delta^2 m ^2+\delta^2\xi^2\lesssim \delta^2\<\lambda\>^2,
\end{align*}
because $|\delta\xi|\leq 1$ for $ \xi \in \mathrm{supp}\psi_{\delta,\lambda}$ due to the constraint of $\lambda$.
Thus, we have $p_\delta''(\xi)\gtrsim \<\lambda\>^{-3}$ for $\xi \in \mathrm{supp}\psi_{\delta,\lambda }$ and by Young's convolution inequality, van der Corput lemma and the estimate \eqref{eq:estQ} combined with the expression \eqref{eq:Uconv} and \eqref{eq:I}, we have
\begin{align}\label{eq:est:first1}
\|U_\delta(t)P_\lambda u\|_{l_\delta^\infty}\lesssim \<\lambda\>^{3/2}t^{-1/2}\|u\|_{l^1_\delta}.
\end{align}
Combining \eqref{eq:est:first1} with the trivial $l^2$ conservation 
$
\|U_\delta(t)P_\lambda u\|_{l_\delta^2}=\|P_\lambda u\|_{l^2_\delta}\lesssim \|u\|_{l^2_\delta},
$
we have by Riesz-Thorin interpolation (see e.g. \cite{BLBook})
\begin{align}\label{eq:qq'est}
\|U_\delta(t)P_\lambda u\|_{l_\delta^q}\lesssim \<\lambda\>^{3(1/2-1/q)}t^{-(1/2-1/q)}\|u\|_{l_\delta^{q'}},
\end{align}
for $q\geq 2$.
Thus, by Bernstein's inequality  (see Lemma 2.3 of \cite{HY19DCDS}) and \eqref{eq:qq'est}, we have
\begin{align}
\|U_\delta(t)P_\lambda u\|_{l^\infty_\delta}&=\|\tilde{P}_\lambda U_\delta(t)P_\lambda \tilde P_\lambda u\|_{l^\infty_\delta}\lesssim \lambda^{1/6}\|U_\delta(t)P_\lambda \tilde{P}_\lambda  u\|_{l^6_\delta}\lesssim \lambda^{1/6}\<\lambda\> t^{-1/3}\|\tilde{P}_\lambda u\|_{l_\delta^{6/5}} \nonumber
\\&
\lesssim \lambda^{1/3}\<\lambda\> t^{-1/3}\|u\|_{l^1_{\delta}}.\label{eq:rep}
\end{align}
This gives \eqref{eq:mainest} for the first case.

Next, we consider the case $1/2\delta< \lambda <2\pi/\delta$.
In this case, since $\lambda\sim \delta^{-1}$, it suffices to prove \eqref{eq:mainest} with $\lambda$ replaced by $\delta^{-1}$.
Since the $\mathrm{supp} \psi_{\delta,\lambda}$ can contain both $\pm \pi/2\delta$ (where $p_\delta''$ degenerates) and $\pm \pi/\delta$ (where $p_\delta'''$ degenerates), we split the integral \eqref{eq:I} into two regions containing only one of $\pm \pi/\delta$ and $\pm \pi/2\delta$ by smooth cutoff.
In the region which do not contain $\pm \pi/\delta$ (and note that because of the constraint of $\lambda$), it will not contain neither $0$, we have $p_\delta'''(\xi)\gtrsim \delta^4$.
Thus, we have 
\begin{align}\label{eq:est:second1}
\|U_\delta(t)P_\lambda u\|_{l_\delta^\infty}\lesssim \(\delta^4 t\)^{-1/3},
\end{align}
by van der Corput lemma, which gives us \eqref{eq:mainest} in this case.
In the region which contains $\pm \pi/\delta$ but not $\pm\pi/2\delta$, we have $p_\delta''(\xi)\gtrsim \delta^3$ so again by van der Corput lemma, we have
\begin{align*}
\|U_\delta(t)P_\lambda u\|_{l_\delta^\infty}\lesssim \(\delta^3 t\)^{-1/2}.
\end{align*}
Repeating the argument of \eqref{eq:qq'est} and \eqref{eq:rep} replacing $\<\lambda\>$ by $\delta^{-1}$ we also have \eqref{eq:est:second1}.
This finishes the proof.
\end{proof}

\begin{proof}[Proof of Theorem \ref{thm:main}]
	By standard $T T^*$ argument (see e.g. \cite{KT98AJM} for general cases and \cite{MSSSS18DCDS} for discrete setting) we obtain
	\begin{align*}
	\| U_\delta(t) P_\lambda u\|_{l^p_\delta l^q_\delta}\lesssim \<\lambda\>^{3/p} \lambda^{1/p}\|u\|_{l_\delta^2},\quad 
	\| \sum_{s\in [0,t]\cap \delta\Z} U_\delta(t-s)P_\lambda f(s) \|_{l^pl^q} \lesssim \<\lambda\>^{3/p +3/\tilde{p}}\lambda^{1/p+1/\tilde{p}}\|f\|_{l_\delta^{\tilde{p}'}l_\delta^{\tilde{q}'}},
	\end{align*}
	for any admissible pairs $(p,q)$ and $(\tilde{p},\tilde{q})$.
	Here, we note that the implicit constants are independent of $\lambda$ and $\delta$.
	Finally, arguing as in Proof of Theorem 1.3 of \cite{HY19DCDS}, we have the desired estimates.
\end{proof}

\section*{Acknowledgments} 
M.M. was supported by the JSPS KAKENHI Grant Number 19K03579, G19KK0066A, JP17H02851 and JP17H02853.

\medskip

Takumi Aso, Masaya Maeda

Department of Mathematics and Informatics,
Faculty of Science,
Chiba University,
Chiba 263-8522, Japan

{\it E-mail Address}: {\tt maeda@math.s.chiba-u.ac.jp, axca4965@chiba-u.jp}

\end{document}